\documentclass[12pt,preprint]{interact}

\usepackage{epstopdf}
\usepackage[caption=false]{subfig}

\usepackage[numbers,sort&compress]{natbib}
\bibpunct[, ]{[}{]}{,}{n}{,}{,}
\renewcommand\bibfont{\fontsize{10}{12}\selectfont}
\makeatletter
\def\NAT@def@citea{\def\@citea{\NAT@separator}}
\makeatother

\def\st{{\rm s.t.}}
\DeclareMathOperator*\argmin{\text{argmin}}

\theoremstyle{plain}
\newtheorem{theorem}{Theorem}[section]
\newtheorem{lemma}[theorem]{Lemma}

\newtheorem{assumption}[theorem]{Assumption}

\theoremstyle{definition}

\theoremstyle{remark}

\newcommand{\bx}{\bm{x}}
\newcommand{\bz}{\bm{z}}
\newcommand{\by}{\bm{y}}
\newcommand{\ba}{\bm{a}}
\newcommand{\bb}{\bm{b}}

\newcommand{\bd}{\bm{d}}
\newcommand{\bg}{\bm{g}}
\newcommand{\bu}{\bm{u}}
\newcommand{\bw}{\bm{w}}
\newcommand{\bv}{\bm{v}}
\newcommand{\be}{\bm{e}}
\newcommand{\la}{\langle}
\newcommand{\ra}{\rangle}
\newcommand{\sign}{\text{sign}}
\newcommand{\Kcal}{\mathcal{K}}

\newcommand{\Hcal}{\mathcal{H}}
\newcommand{\Pcal}{\mathcal{P}}
\newcommand{\Scal}{\mathcal{S}}

\begin{document}

\articletype{ARTICLE TEMPLATE}

\title{Inexact Primal-Dual Gradient Projection Methods for Nonlinear Optimization on Convex Set}

\author{
\name{Fan Zhang\textsuperscript{a, b, e}, Hao Wang\textsuperscript{a}, Jiashan Wang\textsuperscript{c}, and Kai Yang\textsuperscript{d}\thanks{CONTACT: Hao Wang. Email: haw309@gmail.com} }
\affil{\textsuperscript{a}School of Information Science and Technology,
	ShanghaiTech University, 
	Shanghai, 201210, China; \\
	\textsuperscript{b}Shanghai Institute of Microsystem and Information Technology, Chinese
	Academy of Sciences, China;\\
	\textsuperscript{c} Department of Mathematics, University of Washington, USA;\\
	\textsuperscript{d} Department of computer scicence,  Tongji University,  Shanghai, China;\\
	\textsuperscript{e}	University of Chinese Academy of Sciences, Beijing, China.}
}

\maketitle

\begin{abstract}
In this paper, we propose a novel primal-dual inexact gradient projection method for nonlinear optimization problems with convex-set constraint. This method only needs inexact computation of the projections onto the convex set for each iteration, consequently reducing the computational cost for projections per iteration. This feature is attractive especially for solving problems where the projections are computationally not easy to calculate. Global convergence guarantee and $O(1/k)$ ergodic convergence rate of the optimality residual are provided under loose assumptions. We apply our proposed strategy to $\ell_1$-ball constrained problems.  Numerical results exhibit that our inexact gradient projection methods for solving $\ell_1$-ball constrained problems are  more   efficient than the exact  methods.
\end{abstract}

\begin{keywords}
Inexact optimization,  gradient projection methods,  $\ell_1$-ball projection, first-order methods, proximal methods. 
\end{keywords}

\section{Introduction}

The primary  focus of this paper is  on designing, analyzing and testing  numerical methods for 
nonlinear optimization problems with convex-set constraint, which have wide applications in diverse 
disciplines, such as machine learning, statistics, signal processing, and control \cite{boyd1994linear,figueiredo2007gradient,hassibi1999low,ng2004feature,patriksson2008survey,van2008probing}. 
For example, maximum likelihood estimation with sparse constraints~\cite{tibshirani1996regression, Lee2006EfficientLR} often yields norm ball constrained problems.  
Another example is principal component analysis \cite{jolliffe2003modified,sigg2008expectation,witten2009penalized} in statistical analysis and engineering to find simpler or low-dimensional representations for data, and a common approach in this case is to require  each of the generated coordinate to be a weighted combination of only a small subset of the original dimension. This often leads to constraints of different kinds of norm balls.

Gradient Projection Method (GPM) \cite{frank1956algorithm, rosen1961gradient, bertsekas1999nonlinear} is one of the most common approaches for solving convex-set constrained problems. At each iteration, it moves along the direction of the negative gradient, and then projects iterates onto the constraint set if it is outside of the set. GPM only needs to use the first-order derivatives, so it  is  often considered as a scalable solver for  large-scale optimization problems. 
The convergence  of GPM is analyzed in~\cite{combettes2004solving,nesterov2013introductory}, and a linear convergence rate is shown in~\cite{bubeck2015convex} under the assumption of smoothness and strong convexity of the objective.
Many variants have appeared to improve   the original  GPM, and the Spectral Projected Gradient (SPG) method proposed in~\cite{birgin2000nonmonotone,birgin2001algorithm} is one of the most popular variants where the spectral stepsize~\cite{barzilai1988two} and nonmonotone line search \cite{grippo1986nonmonotone} are incorporated for purposes of acceleration and reducing the function evaluations spent on line search. 
It is deemed that GPM is an effective solver for large-scale settings  as long as the projection operations at each iteration  can be carried out  efficiently. However, this may not be the case in many situations as the constraint set could be so complex that the projection can not be easily computed.  
Therefore, this limits the applicability of GPM and also other projection-based algorithms including Nesterov's optimal first-order method \cite{beck2009fast, nesterov2013gradient} and projected Quasi-Newton \cite{schmidt2009optimizing}, which are efficient for problems involving simple constraint sets such as  affine subspace and norm balls.
One way to circumvent this obstacle is to avoid calculating the projections by using Frank-Wolfe method (conditional gradient method) \cite{luss2013conditional,polak1971computational,zangwill1969nonlinear}, but the convergence is shown to be generally slower than GPM~\cite{bertsekas1974goldstein}.

To reduce the computational effort spent on the projections,  inexact gradient projection methods~\cite{birgin2003inexact,patrascu2018convergence} have been proposed to  accept an approximate projection for each iteration.  In this way, the subproblem solver can early terminate once a sufficiently accurate solution to the projection subproblem is reached.   
In Inexact Spectral Projected Gradient (ISPG) method proposed by Birgin et al.~\cite{birgin2003inexact}, the subproblem solver is terminated if the subproblem objective reduction induced by the (feasible) inexact solution is at least a fraction of that induced by the exact projection.  
Another inexact GPM was recently proposed by Patrascu et al.~\cite{patrascu2018convergence}, which requires the subproblem solution to approach the constraint set over the iteration by imposing a sequence of monotonically decreasing feasibility  tolerance for the subproblems. It is shown to converge sublinearly  
 by assuming  Lipschitz differentiability  and convexity of the objective, and to converge linearly  by further assuming  strong convexity.

The key technique in such methods is  the termination criterion for the subproblem solver which has significant effect on the overall performance of GPM. 
If the termination criterion is set to be very strict, meaning the subproblem needs to be solved relatively accurately, then the computational burden of projections may not be alleviated. 
On the other hand, terminating the subproblem too early might slow down the overall convergence.  
Such termination criterion should be practical to implement, and can be easily manipulated to control the progress (namely, the ``inexactness'') achieved by the subproblem solver.  The inexact GPM in \cite{birgin2003inexact} and \cite{patrascu2018convergence} represents two typical kinds of contemporary  inexact termination criteria. 
The subproblem termination criterion in \cite{birgin2003inexact} requires prior knowledge or at least an estimate of the optimal objective.  This can be used as a reliable standard to show ``how far'' the current iterate is from being optimal. However, it is impractical in most cases since the optimal objective is generally unknown before finishing the solve.  Therefore, this method relies on Dykstra's projection algorithm for solving the subproblems as the inexact termination criterion will be eventually satisfied  during the algorithm with some (uncontrollable) tolerance.
The subproblem termination criterion in \cite{patrascu2018convergence} represents a stark contrast to \cite{birgin2003inexact}.  It depends on driving the tolerance for  subproblem optimality residual gradually to zero.  Therefore, the performance of this variant of GPM highly depends on how fast the tolerance converges to zero. And the appropriate  tolerance value may vary much for different problems due to different problem sizes or constraint set structures.

In this paper, we propose a novel primal-dual Inexact Gradient Projection Method (IGPM), which allows inexact solution of the projection subproblem. Unlike other inexact methods, we use the duality gap to evaluate the ``inexactness'' of a subproblem solution. For each projection subproblem solve, we require the gap between the initial objective and current primal value is at least a fraction of that between the initial objective and the dual objective. This condition is further relaxed to be satisfied only within some prescribed tolerance monotonically driven to zero.  Compared with existing methods, this inexact termination criterion does not involve the optimal objective. 
Moreover, the  two gaps  used in this criterion  coincide at optimal primal-dual solution by strong duality, and that the latter gap is  an upper bound for the gap between the initial and optimal primal values. 
Therefore, this  termination criterion clearly  reflects how close the inexact solution is to the optimal solution compared with the starting point. 
This intuition can help the user to easily manipulate the inexactness of the subproblem solve. 
We analyse global convergence and worst-case complexity of the proposed inexact method  for general 
nonlinear Lipschitz-differentiable objectives (not necessarily convex).  
To demonstrate the effectiveness of our proposed strategy, we apply our inexact method  to the $\ell_1$ ball constrained optimization problems.  
We decompose the $\ell_1$ ball projection into a finite number of projections onto hyperplanes, and then use the inexact termination  criterion to reduce this computational effort to only a few hyperplane projections. 

We summarize the novelties of our work below. 
\begin{itemize}
	\item  Our proposed IGPM is especially designed for the cases that the projection can not be computed efficiently.  By reducing the computational effort for each iteration of the IGPM, the overall computational cost can be alleviated. The key technique of this strategy is a novel termination criterion for the projection subproblem  using the duality gap. 
	This criterion need no use of the optimal objective, and clearly reflects how close the current objective is to the optimal objective.  We emphasize that we do not have any requirement for the selection of the projection solver. In fact, the proposed strategy can be incorporated into any projection solver that generates a sequence of convergent primal-dual iterates. 
	
	\item We derive global convergence and worst-case complexity analysis under loose assumptions without requiring the convexity of the objective. We show that the optimality residual has a $O(1/k)$ ergodic convergence rate.   
	
\end{itemize}

\subsection{Organization}
In the remainder of this section, we outline our notation and introduce various concepts that will be employed throughout the paper.  We describe the inexact termination criterion and
our proposed inexact gradient projection method in \S\ref{s.igpm}.  
The analysis of global convergence and worst-case complexity are provided in \S\ref{s.ca}.  
We apply the proposed strategy to solve the $\ell_1$ ball constrained problem   in  \S\ref{s.l1}. 
 The results of numerical experiments are presented in    \S\ref{s.ne}. Concluding remarks are provided  in \S\ref{s.c}.

\subsection{Notation}
Let $\mathbb{R}^n$ be the space of real $n$-vectors and $\mathbb{R}_+$ be the nonnegative orthant of $\mathbb{R}^n$, i.e., 
$\mathbb{R}^n_+ := \{ \bm{x}\in\mathbb{R}^n: \bm{x}\ge 0\}$. 
Define the positive natural numbers set as $\mathbb{N}= \{1,2,3,...\}$.  
The $\ell_1$ norm is indicated as $\|\cdot\|_{1}$ with the $n$-dimensional  $\ell_{1}$
ball with radius $\gamma$ denoted as $\mathbb{B}^n(\gamma):=\{\bm{x}\in\mathbb{R}^n:\|\bm{x}\|_1\leq \gamma\}$. 
 Given a set $\Omega\subset \mathbb{R}^n$, we define the convex indicator for $\Omega$ by 
\[
\delta_\Omega(\bm{x}) =  \begin{cases}  
0, & \text{if } \bm{x} \in \Omega, \\
+\infty, & \text{if } \bm{x} \notin \Omega, 
\end{cases}
\]
and its support function by 
\[ \delta^*_\Omega (\bm{u}) = \sup_{x \in \Omega} \langle \bm{x}, \bm{u} \rangle.\]
If $f$ is convex, then the subdifferential of $f$ at $\bar\bx$ is given by 
\[ \partial f(\bar\bx) : = \{ \bz  \mid   f(\bar\bx) + \la \bz, \bx - \bar\bx\ra \le f(\bx), \forall \bx\in\mathbb{R}^n \} .\]
The subdifferential of the indicator function of $\Omega$ at $\bm x$ is known as the normal cone,
\[
\partial \delta_\Omega(\bm{x})  = \mathcal{N}(\bm{x}|\Omega) = \{ \bm{v}  \mid \langle \bm{v} , \bm{y} - \bm{x} \rangle, \forall \bm{x}\in \Omega   \}.
\]
The Euclidean projection of $\bm{w}$ onto $\Omega$ is given by 
\[ \mathcal{P}_\Omega(\bm{w}) = \argmin_{\bm{x}\in \Omega} \| \bm{x}-\bm{w}\|_2^2. \]   The proximal operator associated with a convex function $h(\bm{x})$ is defined as 
\[
\textup{prox}_{h}(\bm{w})  :=\argmin \limits _{\bm{x}} h(\bm{x})+{\tfrac {1}{2}}\|\bm{x}-\bm{w}\|_{2}^{2}.
\]
It is easily seen that 
$\mathcal{P}_\Omega(\bm{w}) = \textup{prox}_{\delta_\Omega}(\bm{w}).$
Let $\mathop{\sign}(\cdot)$ represent the signum function of a real number $t$, i.e.,
$$\mathop{\sign}(t):= \begin{cases}
1, & \text{if } t>0\\
0, & \text{if } t=0\\
-1, & \text{if }  t<0.
\end{cases} $$ 
For two vectors $\ba$ and $\bb$ of the same dimension, the Hadamard product $\ba\circ\bb$ is defined as  
$(\ba \circ \bb)_i = \ba_i \bb_i$. The absolute value of vector $\ba\in\mathbb{R}^n$ is defined as 
$|\ba| = (|a_1|,\ldots, |a_n|)^T$, and the sign of $\ba$ is denoted as 
$\sign(\ba) = (\sign(a_1), \ldots, \sign(a_n))^T$.  We use $\be$ to represent the vector filled with all 1s  of appropriate dimension, and 
$\bm{0}_n$ to denote the $n$-dimensional vector of all zeros.

\section{Algorithm Description}\label{s.igpm}
In this section, we formulate our problem of interest and outline our proposed inexact methods. 
We present our algorithm in the context of the generic nonlinear optimization problems with convex set constraint 
\begin{equation}\label{f.prob}\tag{NLO}
\begin{aligned}
\min\limits_{\bm{x}\in\mathbb{R}^n} \quad &f(\bm{x})\\
\st \quad&\bm{x}\in\Omega,
\end{aligned}
\end{equation}
where $\Omega \subset \mathbb{R}^{n}$ is a closed convex set and $f\colon \mathbb{R}^{n} \rightarrow \mathbb{R} $ is  differentiable over $\Omega$.

At each iteration, gradient projection method starts from a feasible point, and moves along the negative gradient direction with stepsize $\beta_k > 0$. If the resulted point is not in the convex set, it is projected onto $\Omega$. 
The $k$th iteration can be expressed as 
\begin{equation}
\bm{x}_{k+1} = \mathcal{P}_{\Omega}(\bm{x}_{k} - \beta_k \nabla f(\bm{x}_k)).
\label{f.iter}
\end{equation}
Many options  for selecting $\beta_k$ in~\eqref{f.iter} have appeared. Goldstein~\cite{goldstein1964convex}, and Levitin and Polyak~\cite{levitin1966constrained} show that for 
  $L$-Lipschitz differentiable $f$, global convergence can be guaranteed with $\beta \in (0, 2/L)$. 
McCormick and Tapia~\cite{mccormick1972gradient} suggest exact line search  to enforce sufficient decrease in the objective.
To avoid the   minimization of a piecewise continuously differentiable function, Bertsekas~\cite{bertsekas1974goldstein} suggests a backtracking line search for finding an appropriate   $\beta_k$.  Another well-known approach is the spectral projected gradient (SPG) method \cite{birgin2000nonmonotone,birgin2001algorithm}, where the   stepsize $\beta_k$ is 
selected  based on the Barzilai-Borwein (BB) formula~\cite{barzilai1988two}
\begin{equation}\label{f.bb}
\beta^{BB1}_k = \frac{\bm{s}_{k-1}^T\bm{s}_{k-1}}{\bm{s}_{k-1}^T\bm{y}_{k-1}}  \quad \text{or}\quad \beta^{BB2}_k = \frac{\bm{s}_{k-1}^T\bm{y}_{k-1}}{\bm{y}_{k-1}^T\bm{y}_{k-1}}
\end{equation}
with $\bm{s}_{k-1}=\bm{x}_k-\bm{x}_{k-1}$ and $\bm{y}_{k-1}=\nabla f(\bm{x}_{k})-\nabla f(\bm{x}_{k-1})$.
After obtaining the projection 
\begin{equation}\label{f.spgm}
\bm{z}_k = \mathcal{P}_{\Omega}(\bm{x}_{k} - \beta_k \nabla f(\bm{x}_k)),
\end{equation}
the SPG method further requires sufficient decrease in the objective by (possibly nonmonotone) line search between $\bx_k$ and $\bz_k$ for stepsize $\alpha_k$
\[
\bm{x}_{k+1} = \bm{x}_k + \alpha_k (\bm{z}_k-\bm{x}_k). 
\]
In this paper,  our inexact strategy is presented in the framework of SPG methods. 
For the sake of simplicity,  we use fixed $\beta_k\equiv \beta$ for the description and analysis.  However, we are well aware that our proposed strategy can be easily generalized to other variants where $\beta_k$ is determined by either line search or the truncated BB stepsizes.

We start the description of our inexact method by analyzing the primal and dual of projection subproblems.
At the  $k$th iteration, we use the following shorthand for simplicity 
\[ 
\bm{g}_k: = \nabla f(\bm{x}_k)\quad \text{and} \quad \bm{v}_k := \bm{x}_k - \beta \nabla f(\bm{x}_k).
\]
At $\bx_k$, the problem of projecting $\bm{v}_k$ onto $\Omega$ can be formulated as 
\begin{equation}
\begin{aligned}
\mathop{\min}\limits_{\bm{z}} \ &  p(\bm{z};\bm{x}_k) := \frac{1}{2}\|\bm{z} - \bm{v}_k\|_2^2 + \delta_\Omega(\bm{z}),
\end{aligned}
\label{f.quad}
\end{equation} 
with the optimal solution being denoted as  $\bm{z}_k^*:=\argmin\limits_{\bm{z}}\ p(\bm{z};\bm{x}_k)$. 
For a feasible  $\bm{z}$, let  $\bm d= \bm z - \bm x_k$ and define the  objective reduction  induced by $\bm{z}$ as 
\begin{equation} \label{delta.p}
\Delta p(\bm{z};\bm{x}_k) := p(\bm{x}_k;\bm{x}_k)-p(\bm{z};\bm{x}_k). 
\end{equation}
Notice that a successful direction should give sufficient decrease in the objective, and in particular,  $\Delta  p(\bm{x}_k;\bm{x}_k) = 0$.
The associated Fenchel-Rockafellar dual~\cite{rockafellar2015convex}  of~\eqref{f.quad} is given by
\begin{equation}
\underset{\bm{u}}{\textup{max}}  \ q(\bm{u};\bm{x}_k) :=  -\frac{1}{2}\|\bm{u} - \bm{v}_k\|_2^2 - \delta_\Omega^*(\bm{u}) + \frac{1}{2} \|\bm{v}_k\|_2^2.
\label{f.dualp}
\end{equation}
By weak duality, we have $p(\bz; \bx_k) - q(\bu; \bx_k) \ge 0$ and the equality holds at primal-dual optimal 
solution $(\bz^*, \bu^*)$ by strong duality.

Given an iterative subproblem solver for computing $\bm{z}_k^*$, we use superscript $(j)$ to denote the $j$th iteration of the $k$th subproblem solve.   Assume this solver  is able to generate a sequence of primal-dual pairs $\{(\bm{z}_k^{(j)}, \bm{u}_k^{(j)})\}$ with $\{\bm{z}_k^{(j)}\} \in \Omega$. 
We assume the primal iterates generated by the solver are ``no worse than'' the trivial iterate $\bx_k$ in the sense 
\begin{equation}\label{primal is good}
 p(\bm{x}_k;\bm{x}_k) \ge p(\bm{z}_k^{(j)};\bm{x}_k) .
 \end{equation}
  In other words, it always holds true  that
$\Delta p(\bm{z}_k;\bm{x}_k) \ge 0$.
This is a reasonable requirement, since any iterate violating this condition can be replaced with $\bx_k$. 
It is fixed during the solve of the $k$th subproblem, but will be sequentially cut down to zero during the entire IGPM framework (which will be discussed later).
With these ingredients, we introduce the following important ratio corresponding to the $j$th iterate of the subproblem solver:
 
	\begin{equation}
	\gamma_k^{(j)} := \frac{p(\bm{x}_k;\bm{x}_k)-p(\bm{z}_k^{(j)};\bm{x}_k) + \omega_k }{p(\bm{x}_k;\bm{x}_k)-q(\bm{u}_k^{(j)};\bm{x}_k) + \omega_k}.
	\label{f.ratio}
	\end{equation}
\noindent

We define this ratio to evaluate the progress made by the  $j$th iterate of the subproblem solver, since it can reflect how close the current objective is from the optimal value. 
We make the following observations
\begin{itemize}
	\item First, 
	the denominator is always greater than or equal to the numerator since  $p(\bm{z}_k^{(j)};\bm{x}_k) \ge q(\bm{u}_k^{(j)};\bm{x}_k)$ by weak duality. 
	On the other hand, the numerator is guaranteed to be nonnegative due to \eqref{primal is good}. Therefore, this ratio always satisfies 
	$ \gamma_k^{(j)} \in [0,1]$. 
	
	
	\item For optimal primal-dual  $(\bm{z}_k^*, \bm{u}_k^*)$, we have 
	\[ \frac{p(\bm{x}_k;\bm{x}_k)-p(\bm{z}_k^*;\bm{x}_k) + \omega_k }{p(\bm{x}_k;\bm{x}_k)-q(\bm{u}_k^*;\bm{x}_k) + \omega_k} = 1
	\]	
	by strong duality $p(\bm{z}_k^*;\bm{x}_k) = q(\bm{u}_k^*;\bm{x}_k)$, meaning an exact projection is found. 
	Therefore, as the primal-dual iterates $(\bm{z}_k^{(j)}, \bm{u}_k^{(j)})$   approaches $(\bm{z}_k^*, \bm{u}_k^*)$, we have
	$\gamma_k^{(j)}$   converges  to $1$. 
	
	\item   The inexact method in \cite{birgin2003inexact} uses  the objective reduction induced by the current iterate $\bz_k^{(j)}$ over that caused by 
	the minimizer (the exact projection) to evaluate the progress, i.e.,  
	\begin{equation}\label{old ratio}
	\Delta p(\bm{z}_k^{(j)}) / \Delta p(\bm{z}_k^*).
	\end{equation}
	To see the relationship between ratio \eqref{old ratio} and  $\gamma_k^{(j)}$, note that  
	weak duality implies  $p(\bm{z}_k^*;\bm{x}_k) \ge q(\bm{u}_k^{(j)};\bm{x}_k)$ yielding 
	\[ \Delta p(\bm{z}_k^*)  \le p(\bm{x}_k;\bm{x}_k)-q(\bm{u}_k^{(j)};\bm{x}_k). \]
	Therefore, we know that the proposed ratio $\gamma_k^{(j)}$ satisfies
	\begin{equation}
	\frac{ \Delta p(\bm{z}_k^{(j)}; \bx_k)   + \omega_k }{\Delta p(\bm{z}_k^*; \bx_k) + \omega_k}\ge\gamma_k^{(j)}.
	\label{f.pd}
	\end{equation}
	If we set $\omega_k =0$,  the ratio on the left-hand side of \eqref{f.pd} coincides  with \eqref{old ratio}.
	
	\item   The exactness of the subproblem solution is controlled by the threshold $\gamma$. 
The primary purpose of adding the relaxation parameter  $\omega$ is mainly to prevent  numerical issues when the denominator of $\gamma_k^{(j)}$ becomes tiny around the optimal solution, 
and this parameter can be removed by setting $\omega=0$ in our theoretical analysis and implementation.   We introduce  this parameter 
 since it can allow the acceptance of  a ``more inexact''  subproblem solution and driving this relaxation $\omega_k$ to 0.

\end{itemize}

Observe that   a large value of  $\gamma_k^{(j)}$ close to 1 implies a relatively accurate solution, whereas a small value of  $\gamma_k^{(j)}$ indicates 
the iterate is far from being optimal.  Therefore, we can set a threshold $\gamma \in (0,1]$ for this ratio to let the subproblem be solved inexactly once $\gamma_k^{(j)}$ exceeds the prescribed threshold.  Specifically, given $(\bm{z}_k^{(j)}, \bm{u}_k^{(j)})$, we terminate the projection algorithm if 

\begin{equation}\boxed{
	\gamma_k^{(j)}=\frac{p(\bm{x}_k;\bm{x}_k)-p(\bm{z}_k^{(j)};\bm{x}_k) + \omega_k }{p(\bm{x}_k;\bm{x}_k)-q(\bm{u}_k^{(j)};\bm{x}_k) + \omega_k} \ge \gamma.
}
\label{criterion}
\end{equation}
 \begin{itemize}
 \item By weak duality, if \eqref{criterion} is satisfied, then 
\begin{equation}
\label{eq.temp3}
  \frac{p(\bm{x}_k;\bm{x}_k)-p(\bm{z}_k^{(j)};\bm{x}_k) + \omega_k }{p(\bm{x}_k;\bm{x}_k)-p(\bm{z}_k^*;\bm{x}_k) + \omega_k} \ge \gamma.
  \end{equation}

 \item 
 It should be noticed that the relaxation parameter $\omega_k \ge 0 $  brings more flexibility to our algorithm, making 
 an inaccurate solution more easily accepted. 
  In fact, it can be set $\omega_k=0$,  because   $p(\bm{x}_k;\bm{x}_k) = q(\bm{u}_k^{(j)};\bm{x}_k)$ implies 
 $\bx_k$ is optimal for the $k$-th subproblem.  In this case,   it can be shown later  that $\bx^k$ is also optimal for problem \eqref{f.prob}.  
 \item 	
It  should be clear that we do not have any requirement for the subproblem solver, as long as 
it can generate a sequence of primal-dual iterates converging to the optimal solution so that \eqref{criterion} will be eventually satisfied.  	
\end{itemize}

We are now ready to introduce our  inexact gradient projection method. Algorithm~\ref{a.Inexact} presents a version using constant stepsize $\beta>0$, which is often required to be smaller than $1/L$ as elaborated in later sessions.  
\begin{algorithm}[H]
	\caption{IGPM  without line search}
	\label{a.Inexact}
	\begin{algorithmic}[1] 
		\STATE Given $\varepsilon > 0$, choose $0<\gamma<1$ and $0<\beta\le 1/L$.
		\STATE Initialize $\bm{x}_{0}$,  $\omega_0$ and set $k=0$.
		\REPEAT
		\STATE Solve~\eqref{f.quad} approximately for $\bm{z}_k\in \Omega$ satisfying~\eqref{criterion}.
		\STATE Update $\bm{x}_{k+1} \leftarrow \bm{z}_k$ and $\omega_{k+1}$.
		\STATE Set $k \leftarrow k+1$.
		\UNTIL {$\|\bm{z}_{k}-\bm{x}_{k}\|\le\varepsilon$.}
	\end{algorithmic}
\end{algorithm}

If the Lipschitz constant $L$ of $f$ is impractical to estimate, we impose a backtracking line search method between the current point and the projected point to enforce a sufficient decrease in the objective. In this case, the stepsize $\beta$ actually can be allowed to vary in a prescribed interval \cite{yuan2000truncated}, such as the spectral stepsize where $\beta$ is computed using the BB formulation~\eqref{f.bb}. The iteration of the IGPM with backtracking line search is outlined in Algorithm~\ref{a.Inexact-ls}.

\begin{algorithm}[H]
	\caption{IGPM with line search}
	\label{a.Inexact-ls}
	\begin{algorithmic}[1] 
		\STATE Given  $\varepsilon>0$, choose $0<{\eta}<1$, $0<\gamma<1$, $0<  \alpha\le 1$, $\beta>0$  and $0<\theta<1$.
		\STATE Initialize $\bm{x}_{0}$, $\omega_0$, and set $k=0$.
		\REPEAT
		\STATE Solve~\eqref{f.quad} approximately for $\bm{z}_k\in \Omega$ satisfying~\eqref{criterion} and let $\bd_k = \bz_k - \bx_k$.
		\STATE Let $\alpha_k$ be the largest value in $\{\theta^0  \alpha, \theta^1  \alpha, \theta^2  \alpha, \cdots\}$ such that
		\begin{equation}
		f(\bm{x}_k+\alpha_k{\bm{d}}_k) \le f(\bm{x}_{k}) +\eta\alpha_k\bm{g}_k^T{\bm{d}}_k.
		\label{f.armijo}
		\end{equation}
		\STATE Update $\bm{x}_{k+1} \leftarrow \bm{x}_k + \alpha_k{\bm{d}}_k$ and   $\omega_{k+1}$.
		\STATE Set $k \leftarrow k+1$.
		\UNTIL {$\|\bm{z}_{k}-\bm{x}_{k}\|\le\varepsilon$.}
	\end{algorithmic}
\end{algorithm}

\section{Convergence Analysis}\label{s.ca}
In this section, we provide global convergence and worst-complexity analysis of our proposed methods.  Throughout our analysis, we make the following assumptions about   \eqref{f.prob}. 
\begin{assumption}\label{ass}
	\begin{enumerate}
		\item[(i)] $f$ is bounded below on $\Omega$ by $\underline{f}$, i.e., $f(\bm{x})\ge \underline{f}$ for any $\bm x \in \Omega$. 
		\item[(ii)] $f$ is   $L$-Lipschitz differentiable, meaning there exists a constant $L>0$ such that 
		$$
		\|\nabla f(\bm{x}) - \nabla f(\bm{y})\|_2\le L\|\bm x- \bm y \|_2.
		$$
		for any $\bm x, \bm y \in \Omega$. 
		\item[(iii)] The relaxation sequence  $\{\omega_k\}$ is summable,  i.e.,		
		$\hat{\omega}:=\sum\limits_{k=0}^{+\infty} \omega_k <+\infty.$
	\end{enumerate}
\end{assumption}

\subsection{Global convergence} 
In this subsection, we analyse the global convergence of our proposed two versions of inexact methods.  
 We first provide the following straightforward results about the subproblem   \eqref{f.quad} in the following 
lemma, which can be easily found in many places such as \cite{Lecture_JV}.  
\begin{lemma}\label{l.subp}
	For any $\bx_k \in \Omega$, let $\bz_k^*$ be the optimal solution of the subproblem \eqref{f.quad}. 
	If $\Delta p(\bz;\bm{x}_k)\ge 0$ for $\bz\in \Omega$, then $\bd = \bz - \bx_k$ is a descent direction for $f$ at $\bx_k$. 
	Furthermore, $\bm{x}_k$ is  a first-order optimal solution to  \eqref{f.prob} if and only if $\bz_k^* = \bx_k$.  
\end{lemma}

It is obviously true  that    $\bz_k^* = \bx_k$  is equivalent to 
\begin{equation}\label{f.condition} 
\Delta p(\bz_k^*;\bm{x}_k)=  0,
\end{equation} 
by the strictly convexity of the projection subproblem. 
Therefore, if a point is optimal to the~\eqref{f.prob}, it must satisfy  condition~\eqref{f.condition}.
The first result of our analysis is to show that every  limit point of the iterates generated by our algorithms 
must satisfy this condition.

Also notice that  Lemma~\ref{l.subp} implies that  $\bx_k$  satisfies the first-order optimality condition of \eqref{f.prob} if and only if 
\begin{equation}
\|\bz_k^* - \bx_k \|^2_2 = \|\bm{x}_k - \mathcal{P}_\Omega(\bm{x}_k - \beta \bm{g}_k)\|_2^2 = 0. 
\end{equation} 
Therefore, we can define 
$E(\bm{x}_k;\beta) := \|\bm{x}_k - \mathcal{P}_\Omega(\bm{x}_k - \beta \bm{g}_k)\|_2^2$  as the optimality residual, as is used by 
\cite{Lecture_JV}.  The second result of our analysis 
shows that this optimality residual converges to zero for our proposed methods. 
%

Before proceed to our main result,  we show that the line search stepsize $\{\alpha_k\}$ for Algorithm \ref{a.Inexact-ls} is bounded away from 0 in the following lemma. 
\begin{lemma}\label{l.wd} Suppose $\{\bx_k\}$ is generated by Algorithm \ref{a.Inexact-ls}.  
It holds true that   
\[ \alpha_k \ge  \frac{\theta (1-\eta)}{\beta L}.
\]
\end{lemma}

\begin{proof} 
First of all, by \eqref{primal is good}, we require the subproblem solution always satisfies 
$\Delta p(\bm{z}_k;\bm{x}_k) \ge 0$, or equivalently,  $- \tfrac{1}{2}\|\bm{d}_k\|_2^2 - \beta\bm{g}_k^T\bm{d}_k \ge 0$. 
Hence, we have 
\begin{equation}\label{gd over gd}
- \bm{g}_k^T\bm{d}_k /\|\bd_k\|_2^2  \ge \tfrac{1}{2\beta}.
\end{equation} 

To find a lower bound for $\alpha_k$, it follows from Assumption \ref{ass}(ii) that
\[	f(\bm{x}_k+\alpha\bm{d}_k) \le f(\bm{x}_k)+\alpha\bm{g}_k^T\bm{d}_k + \tfrac{L}{2}\alpha^2\|\bm{d}_k\|_2^2,
\]
Therefore, for any $\alpha \in (0, -\tfrac{2(1-\eta)\bm{g}_k^T\bm{d}_k}{L\|\bm{d}_k\|_2^2}]$, the 
	Armijo condition \eqref{f.armijo} is satisfied. Hence, the backtracking procedure must end up with 
\[ \alpha_k \ge -\tfrac{2\theta (1-\eta)\bm{g}_k^T\bm{d}_k}{L\|\bm{d}_k\|_2^2},\]
yielding 
$\alpha_k \ge  \tfrac{\theta (1-\eta)}{\beta L}$
by \eqref{gd over gd}.

\end{proof}

Next we show that for both algorithms, the subproblem objective reductions induced 
by either the inexact projection $\bz_k$ or the exact projection $\bz^*$ vanish.

\begin{lemma}\label{l.conv1}
	Suppose $\{\bm{x}_k\}$ is generated by Algorithm  \ref{a.Inexact} or  \ref{a.Inexact-ls}. It holds true that 
	\[ \lim\limits_{{k\to\infty}}\Delta p(\bz_k;\bm{x}_k)=0\quad\text{and} \quad \lim\limits_{{k\to\infty}}\Delta p(\bm{z}_k^*;\bm{x}_k)=0.\]
\end{lemma}
\begin{proof}  We first prove this for Algorithm  \ref{a.Inexact} and suppose $\{\bx_k\}$ is generated by Algorithm  \ref{a.Inexact}. Notice that in this case we set
$\bx_{k+1} = \bz_k$. 
It follows from Assumption \ref{ass}(ii) and $\beta \le 1/L$ that 
\begin{equation}
\label{eq.temp1}
	\begin{aligned} 
	f(\bx_{k+1})
	\le &\ f(\bm{x}_{k}) + \bm{g}_k^T(\bx_{k+1}-\bm{x}_{k})+\frac{L}{2}\|\bx_{k+1}-\bm{x}_{k}\|_2^2\\
	\le &\ f(\bm{x}_{k}) + \bm{g}_k^T(\bx_{k+1}-\bm{x}_{k})+\frac{1}{2\beta}\|\bx_{k+1}-\bm{x}_{k}\|_2^2\\
	= &\   f(\bm{x}_{k}) - \frac{1}{\beta}\Delta p(\bx_{k+1};\bm{x}_{k}),
	\end{aligned}
\end{equation}
 then   we have
	$$
	\Delta p(\bx_{k+1};\bm{x}_{k}) \le \beta(f(\bm{x}_{k}) -f(\bm{x}_{k+1})).
	$$
	Summing up both sides of this inequation from $0$ to $t$ gives
\[
	\sum\limits_{k=0}^{t} \Delta p(\bm{x}_{k+1};\bm{x}_{k}) \le \beta\sum\limits_{k=0}^{t}(f(\bm{x}_{k}) -f(\bm{x}_{k+1})) =   \beta(f(\bm{x}_0)- f(\bx_{t+1}) ). 
\]
Letting $t\to \infty$ and from Assumption \ref{ass}(i), we have $\lim\limits_{{k\to\infty}}\Delta p(\bm{x}_{k+1};\bm{x}_k)=0.$ This proves the first limit. 

To prove the second limit,  we have from \eqref{eq.temp3}  that  
	\begin{equation}\label{f.pzpzs}
	\Delta p(\bm{x}_{k+1};\bm{x}_k) + \omega_k\ge \gamma(\Delta p(\bm{z}_k^*;\bm{x}_k) + \omega_k). 
	\end{equation}
Therefore,  $\Delta p(\bm{z}_k^*;\bm{x}_k) \le \frac{1}{\gamma}(\Delta p(\bm{z}_k;\bm{x}_k) + (1-\gamma)\omega_k),$
which, combined with the first limit and  $ \lim\limits_{{k\to\infty}}\omega_k = 0$ by     Assumption \ref{ass}(iii), yields $\lim\limits_{{k\to\infty}}\Delta p(\bm{z}_k^*;\bm{x}_k)=0.$  This completes the proof for Algorithm \ref{a.Inexact}.

Now suppose $\{\bx_k\}$ is generated by Algorithm \ref{a.Inexact-ls}.  We start with the 
Armijo condition \eqref{f.armijo} and sum up its both sides from 0 to $t$ to obtain 
\[ -\sum_{k=0}^t \eta\alpha_k\bm{g}_k^T\bm{d}_k \le \ \sum\limits_{k=0}^t (f(\bm{x}_k)-f(\bm{x}_k + \alpha_k\bm{d}_k)) = f(\bx_0) - f(\bx_{t+1}). \]
Letting $t\to\infty$, we have $ -\sum_{k=0}^\infty \eta\alpha_k\bm{g}_k^T\bm{d}_k  \le f(\bx_0) - \underline f$ by Assumption \ref{ass}(i). 
It follows from Lemma \ref{l.wd} and $\bg_k^T\bd_k \le 0$ that $\bg_k^T \bd_k \to 0$. The fact 
\begin{equation}\label{eq.temp2}
 \Delta p(\bm{z}_k;\bm{x}_k) =  - \beta \bm{g}_k^T\bm{d}_k - \tfrac{1}{2}\|\bd_k\|_2^2 \le   - \beta \bm{g}_k^T\bm{d}_k, 
 \end{equation} 
indicates $ \Delta p(\bm{z}_k;\bm{x}_k) \to 0$. 
It then follows from   \eqref{f.pzpzs} 
that  $ \Delta p(\bm{z}_k^*;\bm{x}_k) \to 0$ is also true.

\end{proof}

We are now ready to provide the first global convergence result for both Algorithms. 
\begin{theorem}\label{global.1.1}
	Suppose $\{\bm{x}_k\}$ is generated by Algorithm  \ref{a.Inexact} or \ref{a.Inexact-ls}. It holds true that 
	every  limit  point $\bm{x}^*$ of $\{\bm{x}_k\}$ must be a first-order optimal solution for \eqref{f.prob}.
\end{theorem}
\begin{proof}	
Based on \eqref{f.condition}, it suffices to show that the subproblem at $\bx^*$ has a stationary point 
$\bz^* = \bx^*$, or equivalently, $\Delta p(\bz^*; \bx^*) = 0$. 
We prove this  by contradiction by assuming 
that  there exists a limit  point $\bm{x}^*$ of  $\{\bm{x}_k\}$  such that 
	\begin{equation}
	\Delta p(\bm{z}^*;\bm{x}^*) = p(\bm{x}^*;\bm{x}^*)-p(\bm{z}^*;\bm{x}^*)\ge\epsilon > 0.
	\label{f.epsa1}
	\end{equation}
Notice by Lemma  \ref{l.conv1},  there exists $k_0 \in\mathbb{N}$ such that 
$\Delta p(\bm{z}_k^*;\bm{x}_k)=p(\bm{x}_k;\bm{x}_k) - p(\bm{z}_k^*;\bm{x}_k)< \epsilon/2$, or equivalently, 
	\begin{equation}
	p(\bm{z}_k^*;\bm{x}_k)  >  p(\bm{x}_k;\bm{x}_k)  - \epsilon/2
	\label{f.eps1a1}
	\end{equation}
	for all $k>k_0$. 
	
To derive a conclusion contradicting  \eqref{f.eps1a1}, first of all, notice that 
\begin{equation}
\label{pxkpxstar} \begin{aligned}
	|p(\bm{x}_k;\bm{x}_k)-p(\bm{x}^*;\bm{x}^*)| &= |\tfrac{1}{2}\beta^2\|\bm{g}_k\|_2^2-\tfrac{1}{2}\beta^2\|\bm{g}_*\|_2^2|\\
	&= \tfrac{1}{2}\beta^2|\|\bm{g}_k\|_2^2-\|\bm{g}_*\|_2^2|\\
	&= \tfrac{1}{2}\beta^2|(\bm{g}_k-\bm{g}_*)^T (\bm{g}_k+\bm{g}_*)|\\
	&\le \tfrac{1}{2}\beta^2\|\bm{g}_k-\bm{g}_*\|_2 \|\bm{g}_k+\bm{g}_*\|_2\\
		&\le \tfrac{1}{2}\beta^2\|\bm{g}_k-\bm{g}_*\|_2 (\|\bm{g}_k\|_2 +\| \bm{g}_*\|_2).
	\end{aligned}
\end{equation}
By Assumption \ref{ass}(ii),  we know there exists $\delta > 0$ such that 
$\| \bg_k \|_2 \le \delta$. It then follows from \eqref{pxkpxstar} that 
\begin{equation}
\label{pxkpxstar2}
 |p(\bm{x}_k;\bm{x}_k)-p(\bm{x}^*;\bm{x}^*)|  \le  \tfrac{1}{2}\beta^2\|\bm{g}_k-\bm{g}_*\|_2 (\delta +\| \bm{g}_*\|_2).
 \end{equation}
Now consider a subsequence    $\Kcal\subseteq \mathbb{N}$ such that   $\{ \bx_k\}_{k\in \Kcal} \to \bx^*$.  By the continuity of $\bg$, there exists $k_1\in \mathbb{N}$ such that 
$\|\bm{g}_k-\bm{g}_*\|_2 \le  \tfrac{\epsilon}{ 2\beta^2(\delta +\| \bm{g}_*\|_2 )}$ for any $k> k_1, k\in\Kcal$. Combining \eqref{pxkpxstar2}, we have 
\begin{equation}\label{eq.ineq1}
  |p(\bm{x}_k;\bm{x}_k)-p(\bm{x}^*;\bm{x}^*)|  \le  \tfrac{1}{2} \beta^2 \frac{\epsilon}{ 2\beta^2(\delta +\| \bm{g}_*\|_2 )}  (\delta +\| \bm{g}_*\|_2) = \epsilon/4. 
\end{equation}

On the other hand, notice that $p(\bz^*; \bx)$ is continuous with respect to $\bx$. Therefore, there exists $k_2\in \mathbb{N}$ such that 
 \begin{equation}\label{eq.ineq2}
	|p(\bm{z}^*;\bm{x}^*) - p(\bm{z}^*;\bm{x}_k)| <\epsilon/4
 \end{equation}
 for any $k>k_2, k\in \Kcal$.

Now let  $k_3 = \max\{k_0, k_1, k_2 \}$.  It following from  \eqref{f.epsa1}, \eqref{eq.ineq1} with  \eqref{eq.ineq2} that 
	\begin{equation}
	\begin{aligned}
	&p(\bm{x}_k;\bm{x}_k)-p(\bm{z}^*;\bm{x}_k)\\
	=&(p(\bm{x}_k;\bm{x}_k)-p(\bm{x}^*;\bm{x}^*))+ (p(\bm{x}^*;\bm{x}^*)- p(\bm{z}^*;\bm{x}^*))+(p(\bm{z}^*;\bm{x}^*) - p(\bm{z}^*;\bm{x}_k))\\
	>& - \epsilon/4 +\epsilon- \epsilon/4 = \epsilon/2
	\end{aligned}
	\label{f.contra1}
	\end{equation}
	for any $k>k_3,\ k\in\mathcal{K}$. This is equivalent to  
	\[p(\bm{z}^*;\bm{x}_k) < p(\bm{x}_k;\bm{x}_k) - \epsilon/2\]
	for any $k>k_3\ge k_0,\ k\in\mathcal{K}$, 
	contradicting \eqref{f.eps1a1} since $\bz_k^*$ is the unique optimal solution of the $k$th subproblem. 
 \end{proof}

We next  investigate  the convergence results of   the optimality residual $E(\bm{x}_k;\beta)$.  
The key technique of this proof relies on showing 
$E(\bm{x}_k;\beta)$ is bounded above by the subproblem objective reduction caused by the exact projection, which 
has been shown to converge to 0.  Before proceeding, we have the following lemma about
the relationship between primal optimal solution $\bm{z}_k^*$  and dual optimal solution  $\bm{u}_k^*$. 

%

\begin{lemma}\label{l.noco} Let 
 $\bm{z}_k^*$ be the optimal solution of  subproblem   \eqref{f.quad} and 
 $\bm{u}_k^*$ be the optimal solution of the dual subproblem  \eqref{f.dualp}.  It then holds true that 
  \[ \bv_k = \bz_k^* + \bm{u}_k^*\quad\text{and}\quad \bm{u}_k^*\in \mathcal{N}(\bm{z}_k^*|\Omega).\]
\end{lemma}
\begin{proof} By the definition of  $\bm{z}_k^*$ and  $\bm{u}_k^*$, we can rewrite 
	\begin{align*}
	\bm{z}_k^* = \argmin \limits _{\bm{z}} \delta_{\Omega}(\bm{z})+{\frac {1}{2}}\|\bm{z}-\bm{v}_k\|_{2}^{2}, \\
	\bm{u}_k^* = \argmin \limits _{\bm{u}} \delta_{\Omega}^*(\bm{u})+{\frac {1}{2}}\|\bm{u}-\bm{v}_k\|_{2}^{2},
	\end{align*}
or, equivalently 
\[\bm{z}_k^* = \text{prox}_{\delta_{\Omega}}(\bm{v}_k)\quad \text{and}\quad \bm{u}_k^* = \text{prox}_{\delta_{\Omega}^*}(\bm{v}_k).\]
	Using the Moreau decomposition  \cite{moreau1965proximite}, we have
	\begin{equation}\label{f.md} 
	\bm{v}_k  = \text{prox}_{\delta_{\Omega}}(\bm{v}_k)+\text{prox}_{\delta_{\Omega}^*}(\bm{v}_k) = \bm{z}_k^* + \bm{u}_k^*.
	\end{equation}
 The optimality condition of the primal subproblem  \eqref{f.quad} gives 
	\[	
	0 \in \bm{z}_k^* - \bm{v}_k +\mathcal{N}(\bm{z}_k^*|\Omega).
	\]
	Combining  this with     \eqref{f.md}, we have $\bm{u}_k^*\in \mathcal{N}(\bm{z}_k^*|\Omega)$.
\end{proof}

We are now ready to analyse the convergence of the optimality residual $E(\bm{x}_k;\beta)$. 

%
%
\begin{theorem}\label{l.KKTcon}
	Suppose   $\{\bm{x}_k\}$ is generated by Algorithm  \ref{a.Inexact} or  \ref{a.Inexact-ls}. It holds that
	\begin{equation}\label{f.KKTcon}
	\lim\limits_{k\to\infty} E(\bm{x}_k;\beta) = 0.
	\end{equation}
\end{theorem}
\begin{proof}
Consider the dual optimal solution $\bu_k^*$. We have from    Lemma  \ref{l.noco} that  $\bm{u}_k^*\in \mathcal{N}(\bm{z}_k^*|\Omega)$, 
meaning 
	\begin{equation}\label{f.normalc}
	(\bm{x}-\bm{z}_k^*)^T\bm{u}_k^* \le 0, \quad \forall \bm{x}\in\Omega.
	\end{equation}
It follows that
	\begin{equation}
	\begin{aligned}
	E(\bm{x}_k;\beta) & = \|\bm{x}_k-\bm{z}_k^*\|^2\\
	&\le (\|\bm{x}_k-\bm{z}_k^*\|^2-2(\bm{x}_k-\bm{z}_k^*)^T\bm{u}_k^*)\\
	&=(\|\bm{x}_k-\bm{z}_k^*-\bm{u}_k^*\|^2-\|\bm{u}_k^*\|^2)\\
	&=(\|\bm{x}_k-\bm{v}_k\|^2-\|\bv_k- \bm{z}_k^*\|^2)\\
	& = 2[ p(\bx_k; \bx_k) - p(\bz_k^*; \bx_k) ]\\
	&=2 \Delta p(\bm{z}_k^*;\bm{x}_k),
	\end{aligned}\label{f.KKTbound}
	\end{equation}
 	where the inequality follows from  \eqref{f.normalc}, the third equality is by Lemma \ref{l.noco} and the fourth equality is from the definition of $p(\cdot ; \bx_k)$. 
Combing this with 	Lemma \ref{l.conv1}, we have 
$\lim\limits_{k\to \infty} E(\bm{x}_k;\beta) = 0$, completing the proof. 
\end{proof}

\subsection{Worst-case complexity analysis}
In this subsection, we provide the worst-case complexity analysis of our proposed IGPM.  It should be noticed that we do not necessarily require the objective 
$f$ to be convex. 
Therefore, we are interested in showing the convergence rate for the ergodic average  of the optimality residual $E(\cdot; \beta)$, where the ergodic average of $E(\bx_t;\beta)$ can be denoted as $\frac{1}{k}\sum\limits_{t=0}^{k-1} E(\bx_t;\beta)$.
 

We first show the complexity result for Algorithm \ref{a.Inexact} in the following theorem. 

\begin{theorem}\label{t.crls}
	Suppose  $\{\bm{x}_k\}$ is generated by Algorithm  \ref{a.Inexact}.  Then 
the ergodic average $\frac{1}{k}\sum\limits_{t=0}^{k-1} E(\bx_t;\beta)$ has  convergence rate $\mathcal{O}(1/k)$.  In particular, 
	\begin{equation}\label{thm.1.complexity}
	\frac{1}{k}\sum\limits_{t=0}^{k-1} E(\bx_t;\beta)  \le \frac{1}{k}\left(\frac{2}{\gamma}(\beta  ( f(\bm{x}_0) -  \underline{f} ) + (1-\gamma) \hat{\omega})\right).
	\end{equation}
\end{theorem}
\begin{proof}
From \eqref{eq.temp1},  we have for $\beta \le 1/L$ that 
	$$
	\begin{aligned}
	f(\bm{x}_{t+1}) -   f(\bm{x}_{t}) &\le - \tfrac{1}{\beta}\Delta p(\bm{x}_{t+1};\bm{x}_{t})\\
	&\le  - \tfrac{\gamma}{\beta}\Delta p(\bm{z}_{t}^*;\bm{x}_{t}) - \tfrac{(\gamma-1)}{\beta}\omega_t\\
	&\le   - \tfrac{\gamma}{2\beta}E(\bx_t;\beta) - \tfrac{(\gamma-1)}{\beta}\omega_t,
	\end{aligned}
	$$
	where the second inequality follows by  \eqref{f.pzpzs} and last inequality follows by  \eqref{f.KKTbound}.
Summing up both sides of the above inequality  from $0$ to $k-1$, we have
	$$
	\begin{aligned}
	\sum\limits_{t=0}^{k-1} (\gamma E(\bx_t;\beta)  + 2 (\gamma-1)\omega_t) &\le 2 \beta\sum\limits_{t=0}^{k-1} (f(\bm{x}_t)-f(\bm{x}_{t+1})) \\
	&=  2\beta ( f(\bx_0) - f(\bx_k))\\
	&\le 2\beta(f(\bm{x}_0) - \underline{f}),
	\end{aligned}
	$$
	by Assumption \ref{ass}(i). 
It follows that 
	$$
	\begin{aligned}
	\sum\limits_{t=0}^{k-1} E(\bx_t;\beta) &\le \tfrac{2}{\gamma}(\beta ( f(\bm{x}_0) -  \underline{f} ) + (1- \gamma) \sum\limits_{t=0}^{k-1} \omega_t )\\
	&\le \tfrac{2}{\gamma}(\beta (f(\bm{x}_0) -  \underline{f} )+ (1-\gamma) \hat{\omega}),
	\end{aligned}
	$$
	where the second inequality follows by  Assumption \ref{ass}(iii). Dividing both sides by $k$,  we know 
	\eqref{thm.1.complexity} is true.
	\end{proof}

%

The complexity analysis of IGPM with backtracking line search is shown in the following theorem.
\begin{theorem}\label{t.cr}
	Suppose  $\{\bm{x}_k\}$ is generated by Algorithm  \ref{a.Inexact-ls}.  Then 
the ergodic average $\frac{1}{k}\sum\limits_{t=0}^{k-1} E(\bx_t;\beta)$ has  convergence rate $\mathcal{O}(1/k)$.  In particular, 
\begin{equation}\label{thm.2.complexity}
  \frac{1}{k} \sum_{t=0}^{k-1}  E(\bx_t;\beta)  \le  \frac{1}{k} \left(  \frac{ 2  (1- \gamma) }{ \gamma} \hat \omega  +   \frac{ 2\beta^2 L}{\eta (1-\eta)\theta \gamma }  (f(\bx_0)  - \underline f) \right).
  \end{equation}

\end{theorem}

\begin{proof}
It holds true that 
 	\begin{equation}
	E(\bx_t;\beta) + 2\omega_t\le 2 (\Delta p(\bm{z}_t^*;\bm{x}_t) + \omega_t)\le\tfrac{2}{\gamma}(\Delta p(\bm{z}_t;\bm{x}_t) + \omega_t)\le \tfrac{2}{\gamma}
	( - \beta\bm{g}_t^T\bm{d}_t + \omega_t),
	\label{f.bla}
	\end{equation}
	where the first inequality is from  \eqref{f.KKTbound}, the second inequality is from \eqref{eq.temp3} and the last inequality is by \eqref{eq.temp2}.
 Therefore, we have 
	\begin{equation}\label{f.gd}
	-\bm{g}_t^T\bm{d}_t \ge   \tfrac{\gamma}{2\beta} E(\bx_t;\beta)  +  \tfrac{\gamma-1}{\beta}\omega_t .
	\end{equation}

		On the other side, by  \eqref{f.armijo} and  Lemma \ref{l.wd}, we have
	\[
	f(\bm{x}_t)-f(\bm{x}_{t+1})\ge-\eta\alpha_t\bm{g}_t^T\bm{d}_t\ge - \frac{\eta(1-\eta)\theta \bm{g}_t^T\bm{d}_t}{ \beta L},
	\]
	 which, together with \eqref{f.gd}, leads to 
	\[
	f(\bm{x}_t)-f(\bm{x}_{t+1})\ge    \frac{\eta(1-\eta)\theta \gamma }{ 2\beta^2 L}   E(\bx_t;\beta)  +  \frac{\eta  (1-\eta)\theta (\gamma-1) }{ \beta^2 L}   \omega_t  .
	\]	
	Summing up both sides from 0 to $k-1$, we have 
	\[
    \frac{\eta(1-\eta)\theta \gamma }{ 2\beta^2 L} \sum_{t=0}^{k-1}  E(\bx_t;\beta)  +  \frac{\eta (1-\eta)\theta (\gamma-1) }{ \beta^2 L}  \sum_{t=0}^{k-1}  \omega_t  \le f(\bx_0) - f(\bx_k)  \le f(\bx_0) - \underline f
	\]	
	by Assumption \ref{ass}(i). 
	Therefore, 
\[   \sum_{t=0}^{k-1}  E(\bx_t;\beta) \le\frac{ 2(1- \gamma) }{ \gamma} \hat \omega  +   \frac{ 2\beta^2 L}{\eta (1-\eta)\theta \gamma }  (f(\bx_0)  - \underline f) \]
Dividing both sides by $k$, we obtain \eqref{thm.2.complexity}. 
   \end{proof}

%

\section{Inexact Projections onto   $\ell_1$ Ball}\label{s.l1}

In this section, we   apply our inexact gradient projection methods  for solving the $\ell_1$ ball constrained problem, where  the set $\Omega$ in  \eqref{f.prob} is an $\ell_1$ ball. Many recent problems in machine learning  \cite{koh2007interior,ng2004feature}, statistics  \cite{candes2007dantzig,tibshirani1996regression}, signal processing  \cite{candes2006stable, Cand2008Enhancing} and compressed sensing  \cite{amelunxen2014living, stojnic2009various} can be formulated into the minimization of  a nonlinear objective on  an $\ell_1$-norm ball,
\begin{equation}
\min \limits_{\bm{x}} \    f(\bm{x}) \quad \text{s.t.} \   \bx \in \mathbb{B}^n(\tau),
\label{f.l1prob}
\end{equation}
where $\tau$ is the radius of the $\ell_1$ ball $\mathbb{B}^n(\tau)$.

We design an active-set method for solving the $\ell_1$-ball projection subproblems, which can be shown to terminate in finite number of iterations for finding the exact projection. For each iteration, this method only needs to compute the projection onto a hyperplane in a space with lower dimension.  
We then incorporate this method into our proposed inexact gradient projection methods, so that 
the occupational cost for each $\ell_1$ ball projection subproblem reduces to only several projections onto the hyperplanes. 
For simplicity, we remove the subscript $k$ from our description as we focus on the projection subproblem, so that 
$p( \cdot )$ and $q(\cdot )$ denote the primal and dual value of the projection problem.

The problem of projecting $\bv$ onto the $\ell_1$-ball can be formulated as 
\begin{equation}
\min \limits_{\bm{z}} \   \tfrac{1}{2} \| \bm{z} - \bm{v}\|_2^2  \quad \text{s.t.} \  \bz\in \mathbb{B}^n(\tau).
 \label{f.l1proj}
\end{equation}
Most existing $\ell_1$-ball projection algorithms take advantage of the symmetry of the $\ell_1$-ball, and transform this projection problem into 
the projection onto  the simplex defined as 
$$
\mathcal{S}(\tau, n):=\{ \bm{x}\in\mathbb{R}^n|\sum_{i=1}^n x_i=\tau\ \text{and}\ x_i\ge0,\forall i=1,2,\cdots,n  \}.
$$
The following lemma, which is presented in  \cite{michelot1986finite, duchi2008efficient}, shows the relationship between the projection onto the $\ell_1$-ball and the simplex. 
\begin{lemma}\label{p.simplex}
	Let $\bm{w}=\mathcal{P}_{\mathcal{S}(\tau, n)}(|\bm{v}|)$, then 
	$$
	\mathcal{P}_{\mathbb{B}^n}(\bm{v})=
	\begin{cases}
	\bm{v},& \text{if}\quad \bm{v}\in\mathbb{B}^n(\tau), \\
	\sign(\bv)\circ \bw,& \text{otherwise}.
	\end{cases}
	$$
\end{lemma}
\noindent
This  lemma allows us to only focus on computing the projection of a vector in the nonnegative orthant $\mathbb{R}^n_+$.  
Therefore, we consider the projection of $\bv \in \mathbb{R}^n_+$ onto $\mathcal{S}(\tau, n)$, which is formulated as 
\begin{equation}
\begin{aligned}
\min \limits_{\bm{w}} \quad & \tfrac{1}{2} \| \bm{w} - \bm{v}\|_2^2 \\
\text{s.t.} \quad & \sum_{i=1}^{n}w_{i} = \tau,\\
& \quad\quad  w_i\geq 0,\quad i=1,\cdots,n.
\end{aligned}\label{f.alter}
\end{equation}
The following lemma \cite{condat2016fast}  is often used to determine the projection onto the simplex. 
\begin{lemma}\label{lem.temp}  There exists a unique $\nu \in \mathbb{R}$ such that the optimal solution of \eqref{f.alter} can be given by
\[ ( \mathbb{P}_{\mathcal{S}(\tau, n)}(\bv))_i = \max(v_i - \nu,0), \quad i=1,\ldots, n.\]
\end{lemma}

Using Lemma~\ref{lem.temp}, many existing algorithms focus on solving the piecewise linear equation 
\begin{equation} \label{piecewise} \sum_{i=1}^n \max(v_i - \nu,0) = \tau \end{equation}
for $\nu$, and then compute the projection accordingly.  
We now show the simplex projection problem \eqref{f.alter} can be further transformed into several projections onto the hyperplanes of the form
\[ \mathcal{H}(\tau, n) := \{ \bm{x} \in \mathbb{R}^n  \mid \bm{e}^T\bm{x} = \tau \}.\]  
Note that the feasible set of  \eqref{f.alter} is the intersection of two convex sets,  the hyperplane $\mathcal{H}(\tau, n)$ and the nonnegative orthant $\mathbb{R}^n_+$. Now consider projecting $\bv$ onto the hyperplane $\mathcal{H}(\tau, n)$. This will result in two cases.

\emph{Case (i)}  All components of 
\[\mathcal{P}_{\mathcal{H}(\tau, n)}(\bv) = \bm{v} - \frac{\bm{e}^T\bm{v}-\tau}{n} \bm{e} \]
 are all nonnegative, then we know $\mathcal{P}_{\mathcal{H}(\tau, n)}(\bv) \in \mathbb{B}(\tau)^n$ and that 
\[\mathcal{P}_{\mathcal{H}(\tau, n)}(\bv)_i = \max(v_i - (\be^T\bv-\tau)/n,0).
\]
By Lemma \ref{lem.temp}, we know 
\[\mathcal{P}_{\mathcal{S}(\tau, n)}(\bv) = \mathcal{P}_{\mathcal{H}(\tau, n)}(\bv).\]
An exact projection onto the simplex is found. 

\emph{Case (ii)}  There exists at least one component $\mathcal{P}_{\mathcal{H}(\tau, n)}(\bv)_j$ of $\mathcal{P}_{\mathcal{H}(\tau, n)}(\bv)$ is negative, i.e., 
$ v_j - (\be^T\bv-\tau)/n < 0$. 
Since 
\[ \sum_{i=1 }^n \mathcal{P}_{\mathcal{H}(\tau, n)}(\bv) = \sum_{i=1 }^n  (v_i - \frac{(\be^T\bv-\tau)}{n}) = \tau,\]
we know 
\[  \sum_{i=1}^n \max(v_i - (\be^T\bv-\tau)/n,0 ) \ge \tau - ( v_j - (\be^T\bv-\tau)/n ) > \tau.  \]
Hence for the unique root $\nu$ of \eqref{piecewise}, it must be true that 
$(\be^T\bv-\tau)/n < \nu$.   A key observation is that in the exact projection onto ${\mathcal{S}(\tau, n)}$, 
the corresponding $j$th component must be zero since we use $\nu > (\be^T\bv-\tau)/n$ to determine the projection by Lemma \ref{lem.temp}. 
Moreover, if we have $v_l - (\be^T\bv-\tau)/n=0$, it must be true that the corresponding $l$th component in the exact projection onto ${\mathcal{S}(\tau, n)}$ must also be zero.

Our projection algorithms are constructed based on these two cases.

%

\subsection{Exact $\ell_1$-ball projection algorithm}

Based on the two cases discussed above, we can project $\bv$ onto $\Hcal(\tau,n)$ and check whether there are negative complements in the projection.  
Given vector $\bw$, denote 
\[ I_+(\bw)  =\{i \mid  w_i>0\},\quad
I_0(\bw)  =\{i \mid  w_i=0\}\quad \text{and} \quad I_- (\bw)  =\{i \mid w_i<0\}.\]
If case (i) happens, meaning $I_-(\Pcal_{\Hcal(\tau,n)}(\bv))=\emptyset$, we should terminate with an exact projection.   If case (ii) happens, meaning 
$I_- (\Pcal_{\Hcal(\tau,n)}(\bv) ) =\{i |w_i<0\} \neq \emptyset$, then we know 
\[ \Pcal_{\Scal(\tau,n)}(\bv)_i = 0,\quad   i\in I_- (\Pcal_{\Hcal(\tau,n)}(\bv)) \cup I_0 (\Pcal_{\Hcal(\tau,n)}(\bv)).\]  
This property makes it possible for us to only focus on the components in $I_+(\Pcal_{\Hcal(\tau,n)}(\bv))$, and eliminate the nonpositive components. After eliminating these components, the simplex projection problem reduces to 
a simplex projection problem in a lower-dimensional space, and the same argument will also apply.  
If we repeat this procedure to obtain the next iterate, then eventually we must encounter   case (i) where all the components of the projection onto the hyperplane 
are nonnegative since $\bv$ is finite dimensional. This procedure is stated in Algorithm \ref{a.l1}, where we use superscript $(j)$ to the $j$th iteration.

%
\begin{algorithm}[H]
	\caption{Exact $\ell_1$-ball projection method} \label{a.l1}
	\begin{algorithmic}[1]
		\STATE Given  $\tau > 0$ and $\bm{v} \notin \mathbb{B}^n(\tau)$. 
		\STATE Initialize $\by^{(0)}= \bw^{(0)}= |\bv|$, $\bw^*=\bm{0}_{n}$ and set $j=0$.
		\REPEAT
		\STATE Calculate $\bm{w}^{(j+1)}  = \bm{y}^{(j)} - \frac{1}{|I_+|}(\bm{e}^T\bm{y}^{(j)} -\tau)\bm{e}$.
		\STATE Update $I_+(\bm{w}^{(j+1)} )$, $I_0(\bm{w}^{(j+1)} ) $ and $I_-(\bm{w}^{(j+1)} )$.
		\STATE Set $\by^{(j+1)} = \bm{w}^{(j+1)}_{I_+(\bm{w}^{(j+1)} )}$.
		\STATE Set $j\gets j+1$.
		\UNTIL $I_-(\bm{w}^{(j)} )=\emptyset$.
		\STATE Set $\bw^*_{I_+(\bm{w}^{(j+1)} )} = \by^{(j+1)}$.
		\STATE Output $\bz^*=\sign(\bv)\circ \bw^*$.
	\end{algorithmic}
\end{algorithm}

In this algorithm, at each iteration, we compute the projection onto a hyperplane in a reduced space of smaller dimension, which needs 
$2|I_+(\bm{w}^{(j+1)} )| $ operations. This 
 can dramatically reduces the computational cost per iteration  if a lot of zero components are detected for each iteration. For the worst case, we   have 
  $| I_-(\bm{w}^{(j+1)} ) |= 1 $ and $| I_0(\bm{w}^{(j+1)} ) |= 0$ for each iteration. In this case, this procedure needs $n$ iterations to terminate, and the computational cost at the $j$th iteration is $2 (n-j)$. 
  Therefore, the worse-case complexity of this algorithm is given by 
  \[ \sum_{j=0}^{n-1} 2(n-j) = n(n+1).\]
  In \cite{condat2016fast}, they consider the similar  algorithm but without reducing the working space for each iteration, and therefore has worst-case complexity of $2n^2$.  However, it is reported in  \cite{condat2016fast} that  $\mathcal{O}(n\log{n})$ complexity  is generally  observed
in practice.

%

\subsection{Inexact $\ell_1$-ball projection algorithm}
%
%
%
As shown in the previous subsection, Algorithm  \ref{a.l1} can be used to find the exact projection onto the $\ell_1$ ball.  Therefore,  the $\ell_1$ ball constraint problem  \eqref{f.l1prob} can be solved by incorporating  Algorithm  \ref{a.l1} within the gradient projection methods. Next we 
design an inexact  version of this projection algorithm to use our inexact gradient projection methods. 

The design of such an inexact  algorithm needs to address two issues.  First, since our IGPM only accepts feasible iterates, we need to guarantee that the subproblem solver returns a point within the $\ell_1$ ball. To achieve this, we add a scaling phase in the projection algorithm  to obtain a feasible 
iterate 
\[ \hat{\bm{w}}^{(j+1)} = \tau\frac{\bm{w}^{(j+1)}}{\|\bm{w}^{(j+1)}\|_1}.\]
It should be noticed that this feasible iterate is only used for computing the ratio $\gamma^{(j)}$, but not for updating the next iterates.  Since the algorithm terminates in finite number of iterations, $\hat{\bm{w}}^{(j+1)}$ eventually converges to the primal optimal solution (in the reduced space).  

On the other hand, for 
   $\Omega=\mathbb{B}^n(\tau)$, the support function can be shown as $\delta_{\Omega}^*=\tau \|\bm{u}^{(j)}\|_{\infty}$. 
   Therefore, the dual of the $\ell_1$-ball projection problem is given by 
\[  \max_{\bu} -\tfrac{1}{2}\|\bm{u} - \bm{v}\|_2^2 - \tau\|\bu\|_{\infty} + \tfrac{1}{2} \|\bm{v}\|_2^2 .\]   
We also need a dual estimate $\bm{u}^{(j)}$  to calculate $q(\bm{u}^{(j)})$ in $\gamma^{(j)}$ without solving the dual problem.    
Note that by   \eqref{f.md}, we have  $\bm{u}^* = \bm{v} - \bm{z}^*$. 
This motivates us to use 
\[ \bm{u}^{(j+1)} = \bm{\bv}^{(j)} - \bw^{(j+1)} \] 
as our dual estimate.  
Therefore, $ \bm{u}^{(j)}$ will eventually converges to the  optimal dual solution (in the reduced space).  Overall, based on our selection of 
$ \hat{\bm{w}}^{(j+1)} $ and $ \bm{u}^{(j+1)} $, we know the corresponding ratio 
$\gamma^{(j)}$ will eventually exceed $\gamma$, and we can terminate the subproblem solver with final point $\hat \bw^{(j+1)}$.

%
\begin{algorithm}[h!]
	\caption{Inexact $\ell_1$-ball projection method} \label{a.il1}
	\begin{algorithmic}[1]
		\STATE Given  $\tau > 0$ and $\bm{v} \notin \mathbb{B}^n(\tau)$. 
		\STATE Initialize $\by^{(0)}= \bw^{(0)}= |\bv|$, $\bw^{(*)}=\bm{0}_{n}$ and set $j=0$.
		\REPEAT
		\STATE Calculate $\bm{w}^{(j+1)}  = \bm{y}^{(j)} - \frac{1}{|I_+|}(\bm{e}^T\bm{y}^{(j)} -\tau)\bm{e}$.
		\STATE Update $I_+(\bm{w}^{(j+1)} )$, $I_0(\bm{w}^{(j+1)} ) $ and $I_-(\bm{w}^{(j+1)} )$.
		\STATE Calculate $\hat\bw^{(j+1)} = \bm{w}^{(j+1)}/\|\bm{w}^{(j+1)}\|_1$ and $\bu^{(j+1)} =  \by^{(j)} - \hat\bw^{(j+1)}$. 
		\STATE Calculate $\gamma^{(j+1)} $ according to \eqref{f.ratio}.
		\STATE Set $\by^{(j+1)} = \bm{w}^{(j+1)}_{I_+(\bm{w}^{(j+1)} )}$.
		\STATE Set $j\gets j+1$.
		\UNTIL $I_-(\bm{w}^{(j)} )=\emptyset$ or $\gamma^{(j)} \ge \gamma$.
		\STATE Set $\bw_{I_+(\bm{w}^{(j)} )}^*=\by^{(j)}$.
		\STATE Output $\bz^*=\text{sign}(\bv)\bw^* $.
	\end{algorithmic}
\end{algorithm}

\noindent 
Applying the 
same argument of Algorithm \ref{a.l1}, we know this algorithm terminates in   at most $n$ iterations.  If we did not terminate this solver prematurely, the ratio $\gamma_k^{(j)}$ would converge to $1$, which coincides with the exact solver. Therefore, we can terminate this solver prematurely when the ratio $\gamma_k^{(j)}$ exceeds the prescribed threshold $\gamma$.

\section{Numerical Experiments}\label{s.ne}

In this section, we apply our proposed methods to sparse signal recovery problem \cite{Cand2008Enhancing}, which aims to recover sparse signal from linear measurements. The sparse signal recovery problem can be formulated as a least squared problem with an $\ell_1$-ball constraint
\begin{equation}\label{f.lasso}
\begin{aligned}
\min_{\bm{x}\in\mathbb{R}^n}\quad & \tfrac{1}{2}\|\bm{A}\bm{x} - \bm{b}\|_2^2\\
\st \quad & \|\bm{x}\|_1 \le \tau,
\end{aligned}
\end{equation}
where $\bm{A}\in \mathbb{R}^{m\times n}$ is the measurement matrix and $\bm{b}\in \mathbb{R}^{m}$ is the observation  vector.

We test both dense and sparse measurement matrix $\bm{A}\in \mathbb{R}^{m\times n}$ cases. Entries of $\bm{A}\in \mathbb{R}^{m\times n}$ is sampled from a Gaussian distribution. 
Denote $s$ as the number of nonzeros of $\bar{\bm{x}}$, where $\bar{\bm{x}}$ is the being recovered signal. We set up experiment as 
\begin{enumerate}
	\item[(a)] Construct $\bar{\bm{x}}\in \mathbb{R}^{n\times 1}$ with randomly choosing $n-s$ components to be zero. Each nonzero entry equals $\pm 1$ with equal probability.
	\item[(b)] Form $\bm{b}=\bm{A}\bar{\bm{x}}$.
	\item[(c)] Solve  \eqref{f.lasso} for $\hat{\bm{x}}$.
\end{enumerate}

Combining our proposed inexact projection framework with the $\ell_1$ projection methods, 
 we have four approaches for solving  the $\ell_1$-ball constrained problem  \eqref{f.l1prob}. 
They are
\begin{itemize}
\item GPM1: exact projection method without line search using Algorithm  \ref{a.l1}.
\item GPM2: exact projection method with line search using Algorithm  \ref{a.l1}.
\item IGPM1:  inexact projection method without line search using Algorithm  \ref{a.il1}. 
\item IGPM2:  inexact projection method with line search using Algorithm  \ref{a.il1}.
\end{itemize}

In the experiments, we initialize $\bm{x}_0 = \bm{0}_n$. The initial value of sequence $\{\omega_k\}$ is chosen as $\omega_0=10^{-3}$. 
The radius $\tau$ of $\ell_1$ norm ball equals $n-s$. For algorithms without line search, the Lipschitz constant of the objective of problem  \eqref{f.lasso} is the largest eigenvalue of the positive semidefinite matrix $\bm{A}^T\bm{A}$, denoted as $\lambda_{\text{max}}$. We calculate $\lambda_{\text{max}}$ approximately by using the power method  \cite{journee2010generalized}, and then we set      $\beta=\frac{0.8}{\lambda_{\text{max}}}$.  For algorithms with backtracking line search, 
we use the parameters 
\[ \beta = 0.01,    \eta = 0.01,    \theta=0.7,    \alpha_0=1.\]
The algorithms  are terminated whenever  
\[ \|\bm{z}_{k}-\bm{x}_{k}\|_\infty  \le 10^{-4}. \]  
Our code is a Python implementation and run on a Dell Precision Tower 7810 Workstation with Intel Xeon processor at 2.40GHz, 64Gb of main memory, and CPU: E5-2630 v3. 

The results are taking average of $20$ runs, which are shown in Table \ref{t.1} to Table \ref{t.6}. The column names are  explained in Table  \ref{t.col}.

\begin{table}[H]
	\caption{Column names}
	\label{t.col}
	\begin{tabular}{c|l}
		\hline
		Alg   & Algorithm used for solving  \eqref{f.lasso} \\ \hline
		$\gamma$   & Different choices of $\gamma=0.6, 0.7, 0.8, 0.9$  \\ \hline
		Time   & CPU time \\ \hline
		Outer $k$ & Number of iterations  of gradient projection  algorithm \\ \hline
		Inner $j$   & Total number of iterations  of projection subproblem solver   \\ \hline
		\#backtracking   &  Total number of backtracking \\ \hline
	\end{tabular}
\end{table}

As shown in Table  \ref{t.1},  \ref{t.2} and  \ref{t.3}, for dense measurement matrix, IGPM1 and IGPM2 outperform GPM1 and GPM2, respectively,  in terms of CPU time. 
The inexact version in most cases does not 
need  more outer iterations to terminate, and it always takes fewer inner iterations, so that the overall computational time is shorter. 
In particular, the inexact version  without line search method in the overdetermined cases ($n\le m$) exhibits  a remarkable superiority for solving  \eqref{f.lasso}.
\begin{table}[h]
	\begin{center}
	\caption{Dense matrix with $n=2\times10^3$, $m=10^4$, $s=10^2$.}
		\begin{tabular}{c|ccccc}
			\hline
			Alg   & $\gamma$  & Time   &Outer $k$ & {Inner $j$} & {\  \#backtracking}\\ 
			\hline
			GPM1    & 1        & $1.12$   & $44.65$ & $189.10$       & $0$    \\ 
			GPM2    & 1        & $3.85$   & $70.95$ & $311.40$       & $331.15$    \\ 
			\hline
			{{IGPM1}} & {$0.6$}  & {$0.81$}  & {$44.60$} & {$117.70$}   &{$0$}  \\ 
			{{IGPM2}}  & {$0.6$}  & {$3.12$}   & {$62.50$} & {$185.30$} & {$294.00$}  \\ 
			\hline
			{IGPM1}    & {$0.7$}   & {$0.82$}  & {$44.70$} & {$122.60$} &{$0$}  \\	
			{IGPM2}    & {$0.7$}   & {$3.00$}   & {$61.05$} & {$193.45$}       &{$289.50$}   \\
			\hline
			{{IGPM1}}  & {$0.8$}  & {$0.81$} & {$44.70$} & {$125.45$}     &{$0$}  \\ 
			{{IGPM2}}  & {$0.8$}  & {$2.74$} & {$5985$} & {$223.50$}   & 	{$286.45$}  \\ 
			\hline
			{{IGPM1}}  & {$0.9$}  & {$0.80$}  & {$44.65$} & {$127.40$}   & {$0$}  \\  
			{{IGPM2}}  & {$0.9$}  & {$2.99$} & {$67.95$} & {$274.40$} & {$317.95$}  \\ 
			\hline
		\end{tabular}\label{t.1}
	\end{center}
\end{table}

\begin{table}[H]
	\caption{Dense matrix with $n=10^4$, $m=10^4$, $s=10^2$.}
	\centering
	\begin{tabular}{c|ccccc}
			\hline
			Alg   & $\gamma$  & Time   &Outer $k$ & {Inner $j$} & {\  \#backtracking}\\ 
			\hline
			GPM1    & 1        & $10.80$       & $90.45$    & $520.95$       & $0$    \\ 
			GPM2    & 1        & $15.30$         & $69.80$ & $393.35$       & $326.35$    \\ 
			\hline
			{{IGPM1}} & {$0.6$}  & {$8.33$}  & {$90.85$} & {$390.75$}   &{$0$}  \\ 
			{{IGPM2}}  & {$0.6$}  & {$11.33$}   & {$58.25$} & {$283.00$} & {$275.75$}  \\ 
			\hline
			{IGPM1}    & {$0.7$}   & {$7.96$}  & {$90.55$} & {$406.65$} &{$0$}  \\	
			{IGPM2}    & {$0.7$}   & {$11.52$}   & {$59.65$} & {$297.70$}       &{$285.65$}   \\
			\hline
			{{IGPM1}}  & {$0.8$}  & {$7.77$} & {$90.50$} & {$417.00$}     &{$0$}  \\ 
			{{IGPM2}}  & {$0.8$}  & {$11.44$} & {$60.50$} & {$304.95$}   & 	{$289.55$}  \\ 
			\hline
			{{IGPM1}}  & {$0.9$}  & {$7.69$}  & {$90.50$} & {$426.40$}   & {$0$}  \\  
			{{IGPM2}}  & {$0.9$}  & {$13.09$} & {$70.75$} & {$367.65$} & {$330.70$}  \\ 
			\hline
	\end{tabular}\label{t.2}
\end{table}

\begin{table}[H]
	\caption{Dense matrix with $n=10^4$, $m=2\times10^3$, $s=10^2$.}
	\begin{center}
		\begin{tabular}{c|ccccc}
			\hline
			Alg   & $\gamma$  & Time   &Outer $k$ & {Inner $j$} & {\  \#backtracking}\\ 
			\hline
			GPM1    & 1     & $35.57$          & $479.50$   & $2874.95$       & $0$    \\ 
			GPM2    & 1     & $5.17$           & $59.00$   & $421.45$       & $89.70$    \\ 
			\hline
			{{IGPM1}} & {$0.6$}  & {$24.75$}  & {$479.80$} & {$2471.30$}   &{$0$}  \\ 
			{{IGPM2}}  & {$0.6$}  & {$3.95$}   & {$69.95$} & {$282.60$} & {$111.45$}  \\ 
			\hline
			{IGPM1}    & {$0.7$}   & {$25.19$}  & {$479.65$} & {$2581.75$} &{$0$}  \\	
			{IGPM2}    & {$0.7$}   & {$3.88$}   & {$68.00$} & {$285.90$}       &{$107.80$}   \\
			\hline
			{{IGPM1}}  & {$0.8$}  & {$25.48$} & {$479.55$} & {$2643.35$}     &{$5$}  \\ 
			{{IGPM2}}  & {$0.8$}  & {$3.83$} & {$62.95$} & {$300.90$}   & 	{$96.55$}  \\ 
			\hline
			{{IGPM1}}  & {$0.9$}  & {$25.67$}  & {$479.55$} & {$2699.80$}   & {$0$}  \\  
			{{IGPM2}}  & {$0.9$}  & {$3.67$} & {$59.30$} & {$299.70$} & {$90.35$}  \\ 
			\hline
		\end{tabular}\label{t.3}
	\end{center}
\end{table}

For the cases of sparse matrix, we increase the dimension of matrix $\bm{A}$, and restrict  the density of the matrix, where density means the number of nonzero-valued elements divided by the total number of elements. We set the density to be $\frac{n}{1000m}$ for all sparse measurement matrix experiments. Due to the high-dimension of matrix $\bm{A}$, calculating the largest eigenvalue of matrix $\bm{A}^T\bm{A}$ becomes impractical. Therefore, we only compare exact and inexact versions with backtracking. Simulation results exhibit that our IGPM2 has a better performance than the GPM2 in saving computational time.

\begin{table}[H]
	\caption{Sparse matrix with $n=10^5$, $m=10^4$, $s=10^4$.}
	\begin{center}
		\begin{tabular}{c|cccccc}
			\hline
			Alg      &$\gamma$  & 	Time   &{Outer $k$} & {Inner $j$} & \#backtracking\\ 
			\hline
			{GPM2}     & 1        & {$28.19$} & {$29.25$} & {$37.3$}       & {$173.8 $}    \\ 
			\hline
			{{IGPM2}}  & {$0.6$}  & {$22.92$} & {$28.3$} & {$14.8$}     & {$169.95$}  \\ 
			\hline
			{IGPM2}    & {$0.7$}   & {$24.52$}  & {$29.3$} & {$20.15$}       & {$174.00$}     \\
			\hline
			{{IGPM2}}  & {$0.8$}  & {$24.41$} & {$29.2$} & {$21.3$}    & {$173.45$}  \\ 
			\hline
			{{IGPM2}}  & {$0.9$}  & {$24.91$} & {$29.2$} & {$24.2$}    & {$173.45$}  \\  \hline
		\end{tabular}\label{t.4}
	\end{center}
\end{table}

\begin{table}[H]
	\caption{Sparse matrix with $n=10^5$, $m=10^5$, $s=10^4$.}
	\begin{center}
		\begin{tabular}{c|cccccc}
			\hline
			Alg    &$\gamma$ & Time  &{Outer $k$} & {Inner $j$} & \#backtracking\\ 
			\hline
			{GPM2} & 1              & {$81.70$}  & {$83.55$} & {$412.75$}       & {$8.35$}    \\ 
			\hline
			{{IGPM2}}  & {$0.6$}  & {$20.84$}  & {$32.05$} & {$56.80$}   & {$10.10$}  \\ 
			\hline
			{IGPM2}    & {$0.7$}   & {$21.60$}  & {$32.60$} & {$62.90$}       & {$10.10$}           \\
			\hline
			{{IGPM2}}  & {$0.8$}  & {$38.90$} & {$57.70$} & {$136.10$}  & {$9.65$}  \\ 
			\hline
			{{IGPM2}}  & {$0.9$}  & {$42.77$}  & {$61.15$} & {$174.15$}  & {$9.35$}  \\  \hline
		\end{tabular}\label{t.5}
	\end{center}
\end{table}

\begin{table}[H]
	\caption{Sparse matrix with $n=10^5$, $m=10^6$, $s=10^4$.}
	\begin{center}
		\begin{tabular}{c|cccccc}
			\hline
			Alg     &$\gamma$ & Time &{Outer $k$} & {Inner $j$} & \#backtracking\\ 
			\hline
			{GPM2}     & 1        & {$102.97$}  & {$96.65$} & {$392.75$}       & {$122.20$}    \\ 
			\hline
			{{IGPM2}}  & {$0.6$}  & {$71.83$}  & {$99.90$} & {$259.30$}     & {$126.00$}  \\ 
			\hline
			{IGPM2}    & {$0.7$}   & {$72.20$} & {$98.30$} & {$275.55$}    & {$124.10$}      \\
			\hline
			{{IGPM2}}  & {$0.8$}  & {$72.02$}  & {$96.65$} & {$283.70$}   & {$121.95$}  \\ 
			\hline
			{{IGPM2}}  & {$0.9$}  & {$73.00$}  & {$96.75$} & {$294.40$}   & {$122.25$}  \\  \hline
		\end{tabular}\label{t.6}
	\end{center}
\end{table}


A key observation in all these experiments is that our algorithms are generally insensitive to the selection of the threshold $\gamma$ when it is varying on a wide range.  In fact, without tuning this parameter carefully, we merely show the performance on these values to avoid the potential slow tail convergence for the subproblem. It indicates that our proposed inexact strategy could achieve stable superior  performance for problems with different sizes and sparsity with the same values of $\gamma$.  To accelerate the local behavior,  one may also use dynamically changing $\gamma$ and drive it to 1 over the iteration. 

\section{Conclusions}\label{s.c}
We have proposed a    framework of  inexact primal-dual gradient projection methods for solving nonlinear problems with convex-set constraint. 
The key technique of such methods is a novel  criterion  to terminate the projection subproblem prematurely, making the methods 
suitable for situations  where the projections onto the convex set are not easy to compute. 
We presented the methods in two cases with or without backtracking line search.  
The   global convergence and  $O(1/k)$ ergodic convergence rate of the optimality residual in worst-case  have  been provided under loose assumptions. 
The proposed methods are applied to $\ell_1$-ball constrained problems, and their performance is exhibited through numerical test on sparse recovery problems. 

Our primary focus here is to reduce the computational cost per iteration of the gradient projection methods so that an inexact projection can be accepted while maintaining the convergence and efficiency. 
We are very aware that there exist more efficient stepsizes for $\beta$, such as the well-known BB stepsizes.  It should be  noticed that our proposed strategy can be easily applied when using truncated BB stepsizes, and the theoretical analysis can be easily generalized to such cases.

\section*{Funding}
This research is supported in part by National Natural Science Foundation of China under grant 61771013,  in part by Jiading Innovation Fund, in part by Research Fund of Tongji University for National Expert.

\bibliographystyle{tfnlm}
\bibliography{ref}

\begin{thebibliography}{10}
\providecommand{\url}[1]{\normalfont{#1}}
\providecommand{\urlprefix}{Available from: }

\bibitem{boyd1994linear}
Boyd~S, El~Ghaoui~L, Feron~E, et~al. Linear matrix inequalities in system and
  control theory. Vol.~15. SIAM; 1994.

\bibitem{figueiredo2007gradient}
Figueiredo~MA, Nowak~RD, Wright~SJ. Gradient projection for sparse
  reconstruction: Application to compressed sensing and other inverse problems.
  IEEE Journal of selected topics in signal processing.
  2007;\hspace{0pt}1(4):586--597.

\bibitem{hassibi1999low}
Hassibi~A, How~JP, Boyd~SP. Low-authority controller design by means of convex
  optimization. Journal of guidance, control, and dynamics.
  1999;\hspace{0pt}22(6):862--872.

\bibitem{ng2004feature}
Ng~AY. Feature selection, $l_1$ vs. $l_2$ regularization, and rotational
  invariance. In: Proceedings of the twenty-first international conference on
  Machine learning; ACM; 2004. p.~78.

\bibitem{patriksson2008survey}
Patriksson~M. A survey on the continuous nonlinear resource allocation problem.
  European Journal of Operational Research. 2008;\hspace{0pt}185(1):1--46.

\bibitem{van2008probing}
Van Den~Berg~E, Friedlander~MP. Probing the pareto frontier for basis pursuit
  solutions. SIAM Journal on Scientific Computing.
  2008;\hspace{0pt}31(2):890--912.

\bibitem{tibshirani1996regression}
Tibshirani~R. Regression shrinkage and selection via the lasso. Journal of the
  Royal Statistical Society Series B (Methodological).
  1996;\hspace{0pt}:267--288.

\bibitem{Lee2006EfficientLR}
Lee~SI, Lee~H, Abbeel~P, et~al. Efficient $l_1$ regularized logistic
  regression. AAAI. 2006;\hspace{0pt}.

\bibitem{jolliffe2003modified}
Jolliffe~IT, Trendafilov~NT, Uddin~M. A modified principal component technique
  based on the lasso. Journal of computational and Graphical Statistics.
  2003;\hspace{0pt}12(3):531--547.

\bibitem{sigg2008expectation}
Sigg~CD, Buhmann~JM. Expectation-maximization for sparse and non-negative pca.
  In: Proceedings of the 25th international conference on Machine learning;
  ACM; 2008. p. 960--967.

\bibitem{witten2009penalized}
Witten~DM, Tibshirani~R, Hastie~T. A penalized matrix decomposition, with
  applications to sparse principal components and canonical correlation
  analysis. Biostatistics. 2009;\hspace{0pt}10(3):515--534.

\bibitem{frank1956algorithm}
Frank~M, Wolfe~P. An algorithm for quadratic programming. Naval research
  logistics quarterly. 1956;\hspace{0pt}3(1-2):95--110.

\bibitem{rosen1961gradient}
Rosen~J. The gradient projection method for nonlinear programming. part ii.
  nonlinear constraints. Journal of the Society for Industrial and Applied
  Mathematics. 1961;\hspace{0pt}9(4):514--532.

\bibitem{bertsekas1999nonlinear}
Bertsekas~DP. Nonlinear programming. Athena scientific Belmont; 1999.

\bibitem{combettes2004solving}
Combettes~PL. Solving monotone inclusions via compositions of nonexpansive
  averaged operators. Optimization. 2004;\hspace{0pt}53(5-6):475--504.

\bibitem{nesterov2013introductory}
Nesterov~Y. Introductory lectures on convex optimization: A basic course.
  Vol.~87. Springer Science \& Business Media; 2013.

\bibitem{bubeck2015convex}
Bubeck~S, et~al. Convex optimization: Algorithms and complexity. Foundations
  and Trends{\textregistered} in Machine Learning.
  2015;\hspace{0pt}8(3-4):231--357.

\bibitem{birgin2000nonmonotone}
Birgin~EG, Mart{\'\i}nez~JM, Raydan~M. Nonmonotone spectral projected gradient
  methods on convex sets. SIAM Journal on Optimization.
  2000;\hspace{0pt}10(4):1196--1211.

\bibitem{birgin2001algorithm}
Birgin~EG, Mart{\'\i}nez~JM, Raydan~M. Algorithm 813: Spg-software for
  convex-constrained optimization. ACM Transactions on Mathematical Software
  (TOMS). 2001;\hspace{0pt}27(3):340--349.

\bibitem{barzilai1988two}
Barzilai~J, Borwein~JM. Two-point step size gradient methods. IMA journal of
  numerical analysis. 1988;\hspace{0pt}8(1):141--148.

\bibitem{grippo1986nonmonotone}
Grippo~L, Lampariello~F, Lucidi~S. A nonmonotone line search technique for
  newton’s method. SIAM Journal on Numerical Analysis.
  1986;\hspace{0pt}23(4):707--716.

\bibitem{beck2009fast}
Beck~A, Teboulle~M. Fast gradient-based algorithms for constrained total
  variation image denoising and deblurring problems. IEEE transactions on image
  processing. 2009;\hspace{0pt}18(11):2419--2434.

\bibitem{nesterov2013gradient}
Nesterov~Y. Gradient methods for minimizing composite functions. Mathematical
  Programming. 2013;\hspace{0pt}140(1):125--161.

\bibitem{schmidt2009optimizing}
Schmidt~M, Berg~E, Friedlander~M, et~al. Optimizing costly functions with
  simple constraints: A limited-memory projected quasi-newton algorithm. In:
  Artificial Intelligence and Statistics; 2009. p. 456--463.

\bibitem{luss2013conditional}
Luss~R, Teboulle~M. Conditional gradient algorithmsfor rank-one matrix
  approximations with a sparsity constraint. SIAM Review.
  2013;\hspace{0pt}55(1):65--98.

\bibitem{polak1971computational}
Polak~E. Computational methods in optimization: a unified approach. Vol.~77.
  Academic press; 1971.

\bibitem{zangwill1969nonlinear}
Zangwill~WI. Nonlinear programming: a unified approach. Vol. 196. Prentice-Hall
  Englewood Cliffs, NJ; 1969.

\bibitem{bertsekas1974goldstein}
Bertsekas~DP. On the goldstein-levitin-polyak gradient projection method. In:
  1974 IEEE Conference on Decision and Control including the 13th Symposium on
  Adaptive Processes; IEEE; 1974. p. 47--52.

\bibitem{birgin2003inexact}
Birgin~EG, Mart{\'\i}nez~JM, Raydan~M. Inexact spectral projected gradient
  methods on convex sets. IMA Journal of Numerical Analysis.
  2003;\hspace{0pt}23(4):539--559.

\bibitem{patrascu2018convergence}
Patrascu~A, Necoara~I. On the convergence of inexact projection primal
  first-order methods for convex minimization. IEEE Transactions on Automatic
  Control. 2018;\hspace{0pt}63(10):3317--3329.

\bibitem{goldstein1964convex}
Goldstein~AA. Convex programming in hilbert space. Bulletin of the American
  Mathematical Society. 1964;\hspace{0pt}70(5):709--710.

\bibitem{levitin1966constrained}
Levitin~ES, Polyak~BT. Constrained minimization methods. USSR Computational
  mathematics and mathematical physics. 1966;\hspace{0pt}6(5):1--50.

\bibitem{mccormick1972gradient}
McCormick~G, Tapia~R. The gradient projection method under mild
  differentiability conditions. SIAM Journal on Control.
  1972;\hspace{0pt}10(1):93--98.

\bibitem{rockafellar2015convex}
Rockafellar~RT. Convex analysis. Princeton university press; 2015.

\bibitem{yuan2000truncated}
Yuan~Y. On the truncated conjugate gradient method. Mathematical Programming.
  2000;\hspace{0pt}87(3):561--573.

\bibitem{Lecture_JV}
Burke~JV. University of \text{Washington}, \text{M}ath 516, \text{L}ecture
  \text{N}otes: The gradient projection algorithm ; 2017.
  \urlprefix\url{https://sites.math.washington.edu/~burke/crs/516/notes/grad-proj-alg.pdf}.

\bibitem{moreau1965proximite}
Moreau~JJ. Proximit{\'e} et dualit{\'e} dans un espace hilbertien. Bull Soc
  Math France. 1965;\hspace{0pt}93(2):273--299.

\bibitem{koh2007interior}
Koh~K, Kim~SJ, Boyd~S. An interior-point method for large-scale l1-regularized
  logistic regression. Journal of Machine learning research.
  2007;\hspace{0pt}8(Jul):1519--1555.

\bibitem{candes2007dantzig}
Cand{\`e}s~E, Tao~T, et~al. The dantzig selector: Statistical estimation when p
  is much larger than n. The annals of Statistics.
  2007;\hspace{0pt}35(6):2313--2351.

\bibitem{candes2006stable}
Cand{\`e}s~EJ, Romberg~JK, Tao~T. Stable signal recovery from incomplete and
  inaccurate measurements. Communications on Pure and Applied Mathematics: A
  Journal Issued by the Courant Institute of Mathematical Sciences.
  2006;\hspace{0pt}59(8):1207--1223.

\bibitem{Cand2008Enhancing}
Cand{\`e}s~EJ, Wakin~MB, Boyd~SP. Enhancing sparsity by reweighted $\ell_1$
  minimization. Journal of Fourier Analysis and Applications.
  2008;\hspace{0pt}14(5-6):877--905.

\bibitem{amelunxen2014living}
Amelunxen~D, Lotz~M, McCoy~MB, et~al. Living on the edge: Phase transitions in
  convex programs with random data. Information and Inference: A Journal of the
  IMA. 2014;\hspace{0pt}3(3):224--294.

\bibitem{stojnic2009various}
Stojnic~M. Various thresholds for $\ell_1$-optimization in compressed sensing.
  arXiv preprint arXiv:09073666. 2009;\hspace{0pt}.

\bibitem{michelot1986finite}
Michelot~C. A finite algorithm for finding the projection of a point onto the
  canonical simplex of $\mathbb{R}^n$. Journal of Optimization Theory and
  Applications. 1986;\hspace{0pt}50(1):195--200.

\bibitem{duchi2008efficient}
Duchi~J, Shalev-Shwartz~S, Singer~Y, et~al. Efficient projections onto the
  $\ell_1$-ball for learning in high dimensions. In: Proceedings of the 25th
  international conference on Machine learning; ACM; 2008. p. 272--279.

\bibitem{condat2016fast}
Condat~L. Fast projection onto the simplex and the $\ell_1$ ball. Mathematical
  Programming. 2016;\hspace{0pt}158(1-2):575--585.

\bibitem{journee2010generalized}
Journ{\'e}e~M, Nesterov~Y, Richt{\'a}rik~P, et~al. Generalized power method for
  sparse principal component analysis. Journal of Machine Learning Research.
  2010;\hspace{0pt}11(Feb):517--553.

\end{thebibliography}

\end{document}